\documentclass[a4paper,10pt]{article}

\title{Odd Hadwiger number and graph products}

\usepackage{amsmath,amssymb,amsthm}
\usepackage{mathrsfs}
\usepackage{bm}
\usepackage{comment}

\usepackage{lmodern}
\usepackage{float}
\usepackage{todonotes}{}
\usepackage[T1]{fontenc}
\usepackage[ mode=buildnew]{standalone}
\usepackage[utf8]{inputenc}
\usepackage{xcolor}
\usepackage{fullpage}
\usepackage{array}
\usepackage{wasysym}
\usepackage{cleveref}

\usepackage{color}
\usepackage{fancyvrb}
\usepackage{tikz}
\usepackage{wasysym}
\usepackage{authblk}
\usepackage{graphicx}
\usepackage{natbib}
\usepackage{amssymb}
\usepackage{comment}
\usepackage{enumerate}

\newtheorem{theorem}{Theorem}[section]

\newtheorem{lemma}[theorem]{Lemma}
\newtheorem{corollary}[theorem]{Corollary}

\newtheorem{claim}[theorem]{Claim}

\newtheorem{problem}[theorem]{Problem}

\usepackage{ulem}

\newenvironment{claimproof}{ \trivlist
	\item[\hskip\labelsep
	\textit{Proof of the claim}.]\ignorespaces
}{\hfill$\vartriangleleft$\medskip
	
}

\author[1]{Henry Echeverr\'ia}
\author[1,2]{Andrea Jim\'enez}
\author[3]{Suchismita Mishra}
\author[1]{Daniel A. Quiroz}
\author[1]{Mauricio Y\'epez}
\date{\today}
	
\affil[1]{\small Instituto de Ingenier\'ia Matem\'atica-CIMFAV, Universidad de Valpara\'iso, Chile.} 

\affil[2]{Millennium Nucleus for Social Data Science (SODAS), Santiago, Chile}

\affil[3]{National Institute of Informatics, Japan}

\usepackage{ulem}

\newcommand{\oddhad}{{\rm{oh}}}
\newcommand{\had}{{\rm{h}}}

\newcommand{\cacher}[1]{}

\begin{document}

\maketitle

\begin{abstract}
The Odd Hadwiger number of a graph $G$ is the largest integer $r$ such that $G$ has a clique of size $r$ as an odd~minor. In this paper, we investigate how large is the Odd Hadwiger number of the product of two graphs, when considering any of the four standard graph products: Cartesian, direct, lexicographic, strong.  We provide an optimal lower bound in the cases of the strong and lexicographic products. 
\end{abstract}


\section{Introduction}

Since Wagner's characterization of planar graphs by excluded minors~\cite{wagner1937}, the theory of graph minors has been a central topic in Graph Theory.  The \underline{Hadwiger number}, $\had(G)$, of a graph $G$ is the maximum integer~$t$ such that $K_t$ is a  minor of $G$. This is a widely studied parameter. In particular, a famous conjecture of Hadwiger \cite{hadwiger1943klassifikation}, which aims at a great generalisation of the Four Colour Theorem, states the following: for every graph $G$, we have $\chi(G)\le \had(G)$. This conjecture holds trivially when $\had(G)\le 2$. Dirac~\cite{dirac1952} showed that any graph $G$ with $\had(G)=3$ is three colourable. By another result of Wagner~\cite{wagner1937}, the Four Colour Theorem implies the case for $\had(G)=4$. Further, by using the Four Colour Theorem, Robertson, Seymour and Thomas \cite{robertson1993hadwiger} verified the case $\had(G)=5$. For larger values, a recent result of Delcourt and Postle \cite{delcourt2025reducing} says that every graph $G$ with $\had(G)\leq t-1$ is $O(t \log \log t)$ colourable. 

A graph $H$ is said to be an \underline{odd~minor} of another graph $G$ if we can obtain $H$ from $G$ by (iteratively) deleting vertices and edges, and contracting all the edges of an edge-cut. The \underline{Odd Hadwiger number}, $\oddhad(G)$, of a graph $G$, is the maximum integer $t$ such that $K_t$ is an odd~minor of $G$. While classes with bounded Hadwiger number are sparse, classes with bounded Odd Hadwiger number can be arbitrarily dense. For instance, it is not hard to see that bipartite graphs exclude $K_3$ as an odd~minor. Indeed, $G$ is bipartite if and only if $\oddhad(G)\le2$.

Odd minors have mostly been considered through a strengthening of Hadwiger's conjecture which was proposed by Gerards and Seymour (see  \cite{jensen2011graph}): for every graph $G$, we have $\chi(G)\le \oddhad(G)$. This conjecture, known as the Odd Hadwiger conjecture, trivially holds for $\oddhad(G)\le 2$, and was proved for $\oddhad(G) = 3$ by Catlin~\cite{C79}. After much work on this conjecture (see e.g.~\cite{EcheverriaJimenezMishraPastineQuirozYepez2025,Geelen,NorinSong,SongThomas2017,S21}), it was recently disproved by  K{\"u}hn, Sauerman, Steiner and Wigderson \cite{kuhn2025disproof}.

In this paper, we investigate the behaviour of the Odd Hadwiger number not in terms of the chromatic number, but in regard to the Cartesian, direct, lexicographic, and strong product of graphs. Our main result is the following.

\begin{theorem}\label{thm:main}
    Let $G$ and $H$ be two graphs with Odd Hadwiger numbers $s$ and $t$, respectively, then $\oddhad(G \ast H) \geq \oddhad(K_s \ast K_t)$, where $\ast$ is either Cartesian, lexicographic or strong product.
\end{theorem}

In fact, for the Cartesian and strong products, we have the following stronger result.

\begin{theorem}\label{thm:strongmain}
    Let $G'$ and $H'$ be odd minors of $G$ and $H$, respectively. Then $G' \star H'$ is an odd minor of $G \star H$, where $\star$ represents Cartesian product or strong product.
\end{theorem}

Note that this result cannot be extended to the direct product since $K_3 \times K_3$ is not an odd minor of $K_3 \times C_5$. 

To our knowledge, our paper is the first one to deal with the relationship between odd minors and graph products. However, there are various papers dealing with the relation between (usual) minors and graph products, especially regarding  Certesian product.  In 1976, Zelinka \cite{Zelinka1976} showed that for every pair of graphs $G$ and $H$ we have $\had(G\Box H)\geq \had(G)+\had(H)-1$.  
Chandran, Kostochka, and Raju \cite{ChandranKostochkaRaju2008} strengthened this result by proving a bound of $\had(G\Box H)\geq \had(G)\sqrt{\had(H)}(1 - o(1))$, which is best possible and gives an asymptotic answer to a question of Miller~\cite{Miller1978}. Wood~\cite{Wood2011} provided various related results and also studied the structure of the Cartesian product of graphs with bounded Hadwiger number.

As an application of Theorem~\ref{thm:main}, we prove the following result, which essentially generalises Zelinka's initial result, save for one vertex.

\begin{theorem}\label{thm:ourzelinka}
    Let $G$ and $H$ be graphs with $\oddhad(G) = s\geq2$ and $\oddhad(H) = t\geq2$, then $\oddhad(G \square H) \geq s+t-2$. 
\end{theorem} 

We do not know if this can be improved, for instance by generalising the result of Chandran, Kostochka and Raju to odd minors. We leave this as an open problem.

\begin{problem}
    Let $G$ and $H$ be graphs with $\oddhad(G) = s$ and $\oddhad(H) = t$. Is it true that $\oddhad (G \square H)\in \Omega(s\sqrt{t})$?
\end{problem}

The strong product case of Theorem~\ref{thm:main} tells us that $\oddhad(G \boxtimes H)\ge \oddhad(G) \cdot \oddhad(H)$, since $K_s \boxtimes K_t= K_{st}$, and thus $\oddhad (K_s \boxtimes K_t)=st$. However, there are graphs $G, H$ for which $\oddhad(G \boxtimes H)$  can be arbitrarily larger than $\oddhad(G) \cdot \oddhad(H)$. The following result tells us that this is true even for stars, and gives an alternative lower bound to that of Theorem~\ref{thm:main} for the lexicographic and strong products. (Here the first inequality comes simply from the fact that $G \boxtimes H$ is a subgraph of $G \Circle H$.)

\begin{theorem}\label{thm:maxdegree}
     For any two graphs $G$ and $H$,
    \begin{equation*}
        \oddhad(G \Circle H) \geq \oddhad(G \boxtimes H) \geq
        \begin{cases}
            \Delta(G)+1, \text{ if } \Delta(G) = \Delta(H)\\
            \min \{\Delta(G), \Delta(H)\}+2, \text{ otherwise.}\\
        \end{cases}
    \end{equation*}
\end{theorem}

While we could not prove a version of Theorem~\ref{thm:main} for the direct product, we conjecture that such a result holds. Moreover, we prove a couple of lower bounds including the following.

\begin{theorem}\label{thm:directbound}Let $G$ and $H$ be graphs  with  $\oddhad(G)\ge s\geq3$ and $\oddhad(H)= 3t\geq9$. 
We have $\oddhad(H\times G)\geq \oddhad( K_{\lfloor s/2\rfloor} \times K_t)$. 
\end{theorem}

To end this introduction, we mention also that our main result is inspired by results in~\cite{collins2023clique, EJQMYim25,WU2025114237} which relate graph products to other notions of containment between graphs, namely immersions and totally odd immersions.

The rest of the paper is organised as follows. In Section~2, we present an equivalent definition for odd minors which we use throughout the paper, as well as the precise definition of each product to be studied. In Section~\ref{sec:main}, we prove that Theorem \ref{thm:strongmain} and use it to prove Theorem~\ref{thm:main}. In Section~\ref{sec:Cartesian} we prove Theorem~\ref{thm:ourzelinka}, and in Section~\ref{sec:lex}, we prove Theorem~\ref{thm:maxdegree}. Finally, in Section 5, we study the behaviour of odd Hawiger number in the direct product, proving, among others, Theorem~\ref{thm:directbound}.

\section{Preliminaries}\label{sec:prel}
All graphs in this paper are finite, simple and loopless. A graph $G$ is called an \emph{expansion} of another graph $H$ if there is a family of vertex-disjoint trees $(T_v)_{v\in V(H)}$ in $G$  such that for every edge $uv \in E(H)$, there exists at least one edge between $T_u$ and $T_v$. It is easy to see that $G$ contains $H$ as a minor if and only if $G$ is an expansion of $H$.

Now, suppose that $G$ is an expansion of $H$ and that there exists a $2$-colouring of $G$ which is a proper colouring in each $T_u, u \in V(H)$, while for any pair of trees $T_u, T_v,$ with $u,v\in V(H)$, at least one edge joining them is monochromatic. Then we say that $G$ is an \emph{odd expansion} of $H$, and that  $G$ contains $H$ as an \emph{odd minor}. Here, the graphs $T_v$ are called the \emph{branch trees} and the colouring $c$ is called the \emph{witness colouring}.

The \emph{lexicographic product} of $G$ and $H$, denoted by $G\Circle H$ is the graph with  vertex set $V(G)\times V(H)$ and where two vertices $(u_1,v_1)$, $(u_2,v_2)$ are adjacent if $u_1u_2\in E(G)$, or $u_1=u_2$ and $v_1v_2\in E(H)$.  The \emph{Cartesian product} of $G$ and $H$, denoted by $G\Square H$, is the graph on vertices $V(G)\times V(H)$ where two vertices $(u_1,v_1),(u_2,v_2)$ are adjacent if $u_1=u_2$ and $v_1v_2\in E(H)$, or $v_1=v_2$ and $u_1u_2\in E(G)$.
The \emph{direct product} (also known as tensor product) of $G$ and $H$, denoted by $G\times H$, is the graph with $V(G)\times V(H)$ as its vertex set and where two vertices $(u_1,v_1),(u_2,v_2)$ are adjacent if $u_1u_2\in E(G)$ and $v_1v_2\in E(H)$. Let $E_1$ and $E_2$ be two sets of edges in $G$ and $H$ respectively. Then $E_1 \times E_2$ is defined as $\{(x,y)(x',y') \mid xx' \in E_1 \text{ and } yy' \in E_2\}$. 
The \emph{strong product} of $G$ and $H$, denoted by $G\boxtimes H$, corresponds to the union $(G\Square H) \cup (G\times H)$; that is, $V(G\boxtimes H)=V(G)\times V(H)$ and two vertices $(u_1,v_1),(u_2,v_2)$ are adjacent if $u_1=u_2$ and $v_1v_2\in E(H)$, or $v_1=v_2$ and $u_1u_2\in E(G)$, or $u_1u_2 \in E(G)$ and $v_1v_2\in E(H)$.  

\section{The product contains the product of the odd minors}\label{sec:main}

In this section we prove Theorem~\ref{thm:strongmain}, and use it to prove Theorem~\ref{thm:main}. By the transitivity of the odd minor relation and the commutativity of the Cartesian and strong products, to prove Theorem~\ref{thm:strongmain}, it suffices to prove the following.

\begin{theorem}
    Let $H'$ be an odd minor of $H$. Then $G \star H'$ is an odd minor of $G \star H$, where $\star$ represents Cartesian product or strong product.
\end{theorem}

\begin{proof}
    Let $\{T_v : v \in V(H')\}$ be a set of trees in $H$ that together with a witness $2$-colouring $c$ forms an odd expansion of $H'$ in $H$. For any $(u,v) \in V(G) \times V(H')$, consider the graph $T'_{(u,v)}$ with vertex set $\{(u,y) \mid y \in V(T_v)\}$ and edge set $\{(u,y)(u,y') \mid yy' \in E(T_v)\}$. Note that $T'_{(u,v)}$ is a tree in both the Cartesian and strong products. Define the $2$-colouring 
    \[c': \bigcup\limits_{(u,v) \in V(G) \times V(H')} T'_{(u,v)} \longrightarrow \{0,1\}\] given by $c'(u,y) := c(y)$. 

    Now we show the above mentioned trees together with the colouring $c'$ form an odd expansion of $G \star H'$ in $G \star H$. First we show that, $c'$ is a proper colouring in each $T'_{(u,v)}$, for  $(u,v) \in V(G) \times V(H')$. Let $(x,y)(x',y')$ be an edge in $T_{(u,v)}$. Then $u = x = x'$ and $yy'$ is an edge in $T_v$. Since $c$ is a proper colouring of $T_v$, from the definition of $c'$ one can ensure that $c'(u,y) \neq c'(u,y')$. Thus, $c'$ is a proper colouring in each $T'_{(u,v)}$.  
    
   Let $(u,v)(u',v')$ be an edge in $G \star H'$. We now need to show the existence in $G \star H$ of a monochromatic edge between the trees $T'_{(u,v)}$ and $T'_{(u',v')}$. Recall that $G \square H'$ is a subgraph  of $G \boxtimes H'$. We first deal with the case $(u,v)(u',v')\in E(G\square H)$, that is, either $u=u'$ and $vv' \in E(H')$ or $v = v'$ and $uu' \in E(G)$. If we have $u = u'$ and $vv' \in E(H')$, then there exists $hh'\in E(H)$ which joins $T_v$ and $T_{v'}$, and is  monochromatic with respect to $c$. Thus by the definition of $c'$ and the trees $T'_{(u,v)}$ and $T'_{(u,v')}$, we know that $(u,h)(u,h')$ is monochromatic with respect to $c'$ and joins  $T'_{(u,v)}$ and $T'_{(u,v')}$.  So we may assume that $v = v'$ and $uu' \in G$. Then the edge $(u,v)(u',v)$joins the trees $T'_{(u,v)}$ and $T'_{(u,v')}$ and is monochromatic under $c'$. In particular, this proves the result when  $\star$ represents the Cartesian product.

To end the proof, we deal with the case in which $(u,v)(u',v')\notin E(G\square H)$, that is, we have $uu' \in E(G)$ and $vv' \in E(H')$.  We know that, there exists a monochromatic edge (with respect to $c$) from $T_v$ to $T_{v'}$ say $yy'$. This tells us that $(u,y)(u,y')$ is monochromatic under $c'$. Since this edge joins $T'_{(u,v)}$ and $T'_{(u',v')}$, the result follows.
\end{proof}

Having proved Theorem~\ref{thm:strongmain} we can now prove Theorem~\ref{thm:main}.

\noindent
\textit{Proof of Theorem~\ref{thm:main}}
    For the Cartesian and strong products, the statement immediately follows from Theorem~\ref{thm:strongmain}. For the lexicographic product, recall that $G \boxtimes H$ is a subgraph of $G \Circle H$, for any two $G$ and $H$. Further, since $K_s \boxtimes K_t= K_s \Circle K_t=K_{st}$, we have  $\oddhad(G \Circle H)\ge\oddhad (G \boxtimes H)\ge\oddhad(K_s \boxtimes K_t)=st=\oddhad (K_s \Circle K_t)$.~\hfill $\square$

\section{An explicit bound for the Cartesian product}\label{sec:Cartesian}
 
In this section, we prove Theorem~\ref{thm:ourzelinka}, which actually follows directly from Theorem~\ref{thm:main} and the following.

\begin{theorem}\label{Theorem:CartesianComplete}

If $s\geq2$ and $t\geq 2$, then $\oddhad(K_s \Square K_t) \geq s+t-2$.
\end{theorem}

\begin{proof}
    Set $V(K_s) := \{u_1,u_2,\dots,u_s\}$ and $V(K_t) := \{v_1,v_2,\dots,v_t\}$  We are looking for a $K_{s+t-2}$ as an odd~minor in $K_s\Square K_t$.
    
    For all $i\in \{2,3,\dots,s\}$, let $S_{(u_i,v_t)}$ be a star of center $(u_i,v_t)$ and leaves $(u_i,v_j)$ for all $j\in \{1,2,\dots,t-1\}$. For each $k\in \{1,\dots s+t-2\}$, we define 
    $$Z_k = \begin{cases}
        (u_1,v_k) &\text{ if } k\in \{1,2,\dots,t-1\}, \\
        S_{(u_{k+2-t},v_t)} &\text{ if } k\in \{t,t+1,\dots,s+t-2\}.
    \end{cases}$$
For every vertex $(u,v)$ in $Z_1 \cup Z_2\cup \ldots \cup Z_{s+t-2}$, we define

$$C((u,v)) = \begin{cases}
    2 &\text{ if } (u,v) \text{ is the center of a star $S_{(u,v)}$,} \\
    1 &\text{ otherwise. }
\end{cases} $$
 See Figure \ref{fig:K_tK_s} for an example of the definition of the $Z_k$'s and the 2-colouring $C$.  

Clearly, $C$ induces a proper 2-colouring on each $Z_k \in \mathbf{Z}$. so it suffices to prove that there is a monochromatic edge between $Z_k$ and $Z_{k'}$ with $1\leq k <k' \leq s+t-2$. This follows from the next three cases. When $1\leq k <k' \leq t-1$, $Z_k$ and $Z_{k'}$ are the singletons $(u_1,v_k), (u_1,v_{k'})$ coloured with 1, hence there is a monochromatic edge between them. When $1\leq k\leq t-1$ and $t \leq k' \leq s+t-2$, $Z_k$ is the singleton $(u_1,v_k)$ with colour 1, and $Z_{k'}$ is a star containing $(u_{k+2-t},v_{k})$ as a leaf with colour 1. Thus the edge $(u_1,v_k)(u_{k+2-t},v_{k})$ is monochromatic. When $t\leq k <k' \leq s+t-2$, $Z_k$ and $Z_{k'}$ are stars with centers $(u_{k+2-t},v_t)$ and $ (u_{k'+2-t},v_t)$, respectively,
    which are coloured with 2 and hence these form a monochromatic edge.
\end{proof}

\begin{figure}[h!]
    \centering
    \includegraphics[scale = 1]{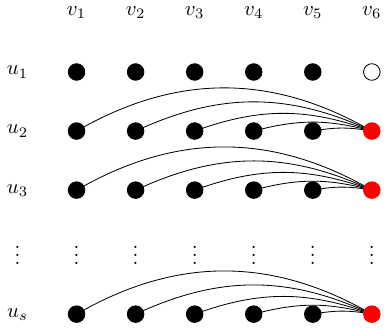}
    \caption{Trees $Z_1, Z_2, \ldots, Z_{s+4}$ and $C$ from the Theorem \ref{Theorem:CartesianComplete} in the case that $t=6$. The colour 1 is represent by black and the colour 2 by red. Note that $Z_1, \ldots, Z_6$ are vertices, the vertex $(u_1,v_6)$ is not used in any tree in $\mathbf{Z}$, and $Z_7, \ldots, Z_{s+4}$ are stars with 5 leaves. }
    \label{fig:K_tK_s}
\end{figure}

We apply Theorem \ref{Theorem:CartesianComplete}  to Hamming graphs. Given $d,q,\in \mathbb{N}$, the the \underline{Hamming graph} $H_t^d$ is the Cartesian product of $d$ complete graphs $K_t$. The next corollary follows by induction.

\begin{corollary} For all $d,q \in \mathbb{N}$, we have
    $\oddhad(H_n^d) \geq d(n-2)+2$.
\end{corollary}

\section{Alternative bound for the lexicographic and strong products}\label{sec:lex}

In this section we prove Theorem~\ref{thm:maxdegree}.

Let $S_k$ be the star with $k$ leaves. Note that each star has Odd Hadwiger number $2$, and hence by Theorem \ref{thm:main} the product of two stars is at least $4$. In Lemma \ref{lemma:stars}, we give a lower bound for the Odd Hadwiger number of the product of two stars which depends on the maximum degree of these. This result immediately implies Theorem~\ref{thm:maxdegree}.

\begin{lemma}\label{lemma:stars}
    For any two integers $r,t\ge 1$ we have
    \begin{equation*}
        \oddhad(S_r \boxtimes S_t) \geq
        \begin{cases}
            r+1, \text{ if } r=t\\
            \min \{r,t\} +2, \text{ otherwise.}\\
        \end{cases}
    \end{equation*}  
\end{lemma}

\begin{proof}
    Let $S_r$ and $S_t$ be the stars with central vertices $u_0$ and $v_0$, respectively. Suppose $\{\cacher{u_0,}u_1, \dots, u_r\}$ and $\{\cacher{v_0,}v_1, \dots, v_t\}$ are the sets of leaves in $S_r$ and $S_t$, respectively. Without loss of generality, we may assume that $t \geq r$. We first construct an odd expansion of $K_r$ in $S_r \boxtimes S_t$. That is, we construct a collection of $r$ vertex disjoint trees, say $\mathbf{Z} = \{Z_1,Z_2,\cdots,Z_r\}$ together with a 2-colouring $c$ on $V(Z_1) \cup V(Z_2) \cup \dots V(Z_r)$ such that $c$ is a proper 2-colouring on each tree in $\mathbf{Z}$ and there is a monochromatic edge between each  pair of distinct trees in $ \mathbf{Z}$. Then we extend this odd expansion by adding a single-vertex tree together with an appropriate colour. When $t>r$, we further expanding this odd expansion by adding another single-vertex tree with appropriate colouring.
    
    Consider the paths $Z_i :=(u_i,v_0),(u_i,v_i),(u_0,v_i)$, for all $1 \leq i \leq r$. For all $1 \leq i \leq r$, we define $c$ as follows: we colour the central vertex $(u_i,v_i)$ with colour $1$ and the vertices $(u_i,v_0)$ and $(u_0,v_i)$ with colour~$2$. This is a proper 2-colouring on $Z_i$. Further, note that for any $1 \leq i < i' \leq r$, the vertices $(u_i,v_0)$ and $(u_0,v_{i'})$ are adjacent. Since both had receive colour~2, this edge is monochromatic edge connecting $Z_i$ with $Z_{i'}$. Now we extend this collection by adding $(u_0,v_0)$ as a new tree and assign colour $2$ to $(u_0,v_0)$.  Since $(u_0,v_0)$ is adjacent to $(u_i,v_0)$, for all $1 \leq i \leq s$, this provides an odd expansion of~$K_{r+1}$. 
    
    We extend this odd expansion by one, when $t>r$.  Consider the vertex $(u_0,v_t)$ and colour it with colour $2$. For any $1 \leq i \leq r$, there exists an edge with endpoints $(u_i,v_0)$ and $(u_0,v_t)$, which is monochromatic under the extension of $c$. 
    Further, $(u_0,v_0)$ is adjacent to $(u_0,v_t)$. Hence, the product graph admits $K_{r+2}$ as an odd~minor.
\end{proof}

\section{Direct Product}\label{sec:direct}

While we were not able to extend our main result, Theorem~\ref{thm:main}, to the direct product, in this section we prove Theorem~\ref{thm:directbound} which gives a weaker bound for the odd Hadwiger number of the direct product of two graphs. We also prove a lower bound for the direct product of two complete graphs, which together with Theorem~\ref{thm:directbound} gives an explicit lower bound on the odd Hadwiger number of any two graphs.

It is well known that $\chi(G\times H)\le\min\{\chi(G), \chi(H)\}$. Moreover, a graph $G$ is bipartite if and only if $\oddhad(G)=2$. Thus we can focus on the direct product of graphs with odd Hadwiger number at least 3.

\subsection{Direct product of complete graphs}\label{sub:2Ks}

We start by studying $\oddhad(K_t \times K_3)$, as this is the basis for our argument when we take any two cliques. Moreover, in this case we get a tight bound.

\begin{theorem}\label{KtK_3}
   For $t\geq 6$, $\oddhad(K_t\times K_3)= t+2$
\end{theorem}

\begin{proof}
Suppose $t\geq 6$, let $\{u_1,u_2,u_3,\dots,u_t\}$ be the vertices of $K_t$ and $v_1,v_2,v_3$ be the vertices of $K_3$. 

    \paragraph{Lower Bound:} We are looking for an odd expansion of $K_{t+2}$. If $t\geq 7$ we define trees $Z_1,Z_2,\ldots,Z_{t+2}$ according to Table \ref{table:trees1}. The third column of Table \ref{table:trees1} contains a partial description 
    of the 2-colouring $c$ of $K_t\times K_3$,
    which is proper on every $Z_i$: the colour of exactly one vertex of each tree $Z_i$ is prescribed, and the remaining vertices are coloured according to the given prescription, ensuring that each tree $Z_i$ receives a proper colouring. 

    \begin{table}[h!] 
        \centering
        \begin{tabular}{|c|c|c|}\hline
           $i$  &  $Z_i$  & $c$ is such that\\ \hline
           1  &$(u_1,v_1)(u_2,v_2)$  & $c((u_1,v_1)) = 1$ \\\hline
           2  &$(u_2,v_1)(u_3,v_2)$  & $c((u_2,v_1)) = 2$\\\hline
           3  &$(u_1,v_3)(u_3,v_1)$  & $c((u_1,v_3)) = 1$\\\hline
           4  &$(u_3,v_3)(u_4,v_1)$  & $c((u_4,v_1)) = 1$\\\hline
           5  &$(u_4,v_2)(u_5,v_3)$  & $c((u_5,v_3)) = 1$\\\hline
           6  &$(u_5,v_2)(u_6,v_3)$  & $c((u_5,v_2)) = 2$\\\hline
           7  &$(u_{1},v_{2})  (u_{2},v_{3})  (u_{5},v_{1})$  & $c((u_{2},v_{3})) = 1$\\\hline
           8  &$(u_{7},v_1)(u_{4},v_3)(u_6,v_2)$  & $c((u_{4},v_3)) = 2$ \\\hline
          $ 9\leq i \leq t+1,~i$ odd &$(u_{i-1},v_2)(u_{i-3},v_1)(u_{i-2},v_3)$  & $c((u_{i-3},v_1)) = 2$\\\hline
           $ 10\leq i \leq t+1,i$ even  &$ (u_{i-1},v_1)(u_{i-3},v_2)(u_{i-2},v_3)$  & $c((u_{i-3},v_2)) = 2$ \\\hline
           $t+2$ odd&$(u_{t},v_2)(u_{t-1},v_1)(u_{t},v_3)$& $c((u_{t-1},v_1)) = 2$\\\hline
           $t+2$ even & $(u_t,v_1)(u_{t-1},v_2)(u_t,v_3)$ & $c((u_{t-1},v_2)) = 2$\\\hline
        \end{tabular}
        \caption{Trees $Z_1,\dots , Z_{t+2}$ (all of them paths) and a partial description of the colouring $c$. }
        \label{table:trees1}
    \end{table}

For the special case of $t=6$, the trees $Z_1, Z_2, \ldots, Z_7$ are defined according to Table \ref{table:trees1}, and tree $Z_8$ is the path $(u_6,v_1)(u_4,v_3)(u_6,v_2)$ with $c((u_4,v_3))=2$ (See  Figure \ref{fig:ktk3}).

    \begin{figure}[H]
  \centering
  \begin{minipage}[t]{0.4\textwidth}
    \centering
    \includegraphics[scale=1]{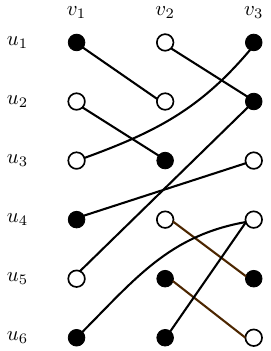}
    \caption{Case $t=6$. Recall that $Z_8$ is defined as $(u_6,v_1)(u_4,v_3)(u_6,v_2)$.}
  \end{minipage}
  \hspace{0.1\textwidth} 
  \begin{minipage}[t]{0.4\textwidth}
    \centering
    \includegraphics[scale=1]{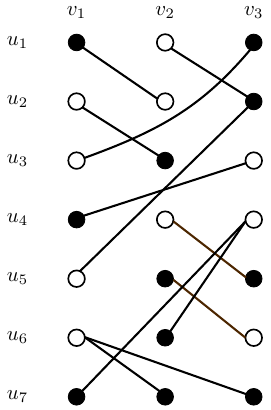}
    \caption{Case $t=7$. Tree $Z_9$ is obtained from the row ``$t+2$ odd'' in Table \ref{table:trees1}. }
  \end{minipage}

        \label{fig:ktk3}
    \end{figure}  

We now establish the existence of  a monochromatic edge between each pair in $Z_1,\dots, Z_{t+2}$.

First note that for trees $Z_i$, $Z_j$ with  $i,j \in \{9,\dots , t+2\}$, if $i, j$ are both odd or both even, then there is a monochromatic edge (with both vertices coloured 1) between end vertices of the paths $Z_i$ and $ Z_j$. If $i$, $j$ are of distinct parity, then there is a monochromatic edge (with both vertices coloured 2) between the inner vertices of $Z_i$ and $Z_j$.

Let us now deal with pairs of trees in $Z_1, \ldots, Z_8$. In Table 2 we give a monochromatic edge joining every such pair of trees. We note that for the case $t=6$, the same edges described in Table \ref{table:mono1-8} witness the monochromatic edges between trees $Z_1, \ldots, Z_7, Z_8$  except for the connection between $Z_8$ and $Z_6$ since vertex $(u_7,v_1)$ does not exists in $K_6 \times K_3$, but in this case, we consider the edge $(u_5,v_2)(u_6,v_1)$, which is monochromatic since both vertices are coloured 1. 

 \begin{table}[h!] 
        \centering
        \begin{tabular}{|c|c|c|c|c|c|c|c|}\hline
             &  $Z_2$  & $Z_3$ & $Z_4$ & $Z_5$ & $Z_6$ & $Z_7$ & $Z_8$ \\ \hline
           $Z_1$  &$(1,1)(3,2)$ & $(2,2)(3,1)$ &$(2,2)(3,3)$&$(1,1)(5,3)$&$(1,1)(5,2)$&$(1,1)(2,3)$&$(1,1)(6,2)$\\\hline
           $Z_2$ &- &$(3,2)(1,3)$ &$(2,2)(3,3)$  &$(2,2)(4,2)$&$(2,2)(6,3)$&$(2,2)(1,2)$&$(2,2)(4,3)$\\\hline
           $Z_3$   & - & - &$(1,3)(4,1)$&$(3,1)(4,2)$&$(1,3)(5,2)$& $(3,1)(1,2)$ &$(1,3)(6,2)$\\\hline
           $Z_4$  & - & - & - &$(4,1)(5,3)$ & $(4,1)(5,2)$ & $(4,1)(2,3)$ & $(4,1)(6,2)$\\\hline
           $Z_5$  & - & - & - & - & $(4,2)(6,3)$ &$(4,2)(5,1)$& $(5,3)(6,2)$ \\\hline
           $Z_6$  & - & - & - & - & - &$(6,3)(5,1)$ & $(5,2)(7,1)$\\\hline
           $Z_7$  & - & - & - & - & - &- & $(5,1)(4,3)$\\\hline
          
        \end{tabular}
        \caption{Monochromatic edges between trees $Z_1, \ldots, Z_8$. In the table, vertex $(u_i,v_j)$ is denoted by $(i,j)$.      
        In the case  $t=6$, instead of $(5,2)(7,1)$, the edge $(5,2)(6,1)$ is used.}
        \label{table:mono1-8}
    \end{table}

Now, note that each $Z_1, Z_4$ contains a vertex of the form $(u_i,v_1)$, for some $i\leq 4$, which is coloured 1, and that each $Z_2, Z_6, Z_8$ contains a vertex of the form $(u_i, v_2)$, for some $i\leq 6$, again coloured~1. Furthermore, each path $Z_9, \ldots, Z_{t+2}$ has an end vertex $(u_j, v_3)$ for some $j\geq 7$ coloured~1.  Thus, there are monochromatic edges between each $Z_1, Z_2, Z_4, Z_6, Z_8$, and each of the paths $Z_9, \ldots, Z_{t+2}$.

 We are only missing the monochromatic edges joining $Z_3, Z_5, Z_7$ to the paths $Z_9, \ldots, Z_{t+2}$. Each of $Z_3, Z_5, Z_7$ contains vertex of the form $(u_i,v_3)$, with $i\le 5$, which is coloured 1, and each tree in $Z_9, \ldots, Z_{t+2}$ contains either a vertex of the form $(u_j, v_1)$, or $(u_j, v_2)$ with $j\geq 8$ coloured~1 as well. We conclude that $K_t\times K_3$ contains an odd expansion of $K_{t+2}$, as desired.

\paragraph{Upper Bound:} 
Consider a largest odd expansion of a complete graph in $K_t \times K_3$. Let $D$ denote the number of trees isomorphic to $K_2$, and let $S$ denote the number of singletons used in the expansion. It is not hard to see that we have $S\leq 3$. Since every remaining tree in the expansion must contain at least three vertices, the total number of trees is at most
\begin{equation}
    \frac{3t - 2D - S}{3} + D + S. \label{nombre:comp}
\end{equation}
This, together with the following claim, will easily give the desired upper bound.

    \begin{claim} We have $D\leq 6-2S$.\label{claim:double}
    \end{claim}
    \begin{claimproof}
Let $V(K_3)=\{v_1,v_2,v_3\}$. 
We classify the edges of $K_t \times K_3$ according to the second coordinate of its endpoints: those of the form $(u,v_1)(u',v_2)$, or the form $(u,v_2)(u',v_3)$, or the form $(u,v_1)(u',v_3)$, with $u,u' \in V(K_t)$. 
We refer to these edges as type $1$, $2$, or $3$, respectively.
Note that there are at most two edges of each type in $D$, since otherwise there would exist a pair of trees in $D$ with no monochromatic connection. In particular, $D \leq 6$.

Without loss of generality assume that the singletons in $S$ are all coloured with colour $1$. Consider a singleton in $S$ of the form $(u,v_i)$, with $i \in \{1,2,3\}$; first suppose $i=1$.  Note that having $(u,v_1)\in S$, tells us that no edge in $D$ can have an endpoint with second coordinate $v_1$ and colour 1. Thus $(u,v_1)\in S$ forbids the existence of an edge of type $1$ and an edge of type $3$. The same argument holds, when $i=2,3$. Moreover, since the sets of edges forbidden by distinct singletons in $S$ are pairwise disjoint, we conclude that $D \leq 6 - 2S.$
\end{claimproof}


Evaluating the bound from Claim \ref{claim:double} in \eqref{nombre:comp}, we obtain 
$$\frac{3t - 2D - S}{3} + D + S \leq t+2.$$
which gives the required bound for the total number of trees in the odd expansion.
\end{proof}

Using $\lfloor s/3 \rfloor$ pairwise disjoint triangles in $K_s$, we extend the result of Lemma~\ref{KtK_3} as follows. 

\begin{theorem}
    \label{thm:complete_direct}
For each $t\geq 4$ and $s\geq 3$, we have $\oddhad(K_t\times K_s)\geq t\lfloor s/3\rfloor$. 
\end{theorem}
\begin{proof} As it suffices, we only consider the case $s\equiv 0~(\rm mod~3)$.
    Let $\{u_1, u_2, \ldots,u_t\}$ and $\{v_1,v_2,\ldots,v_s\}$ be the vertex sets of $K_t$ and $K_s$, respectively.  
    We decompose $K_s$ into $s/3$ triangles, as follows.  For each $\ell \in \{1,2,\dots,s/3\},j\in\{1,2,3\}$, denote $v_{j}^\ell:=v_{j+3(\ell-1)}$. That is, $v_1^\ell,v_2^\ell,v_3^\ell$ are the vertices in the $\ell$-th triangle in $K_s$.
    Our strategy is as follows. First, we construct an odd expansion of $K_t$ in the graph $K_t \times K_s[v_1^1,v_2^1,v_3^1]$. Then, for each $\ell \geq 2$, we construct an odd expansion of $K_{t-2}$ in the graph $K_t \times K_s[v_1^\ell,v_2^\ell,v_3^\ell]$. Finally, for each $2 \leq \ell \leq s/3$, we construct two additional trees that take vertices in both the $(\ell-1)$-th triangle and the $\ell$-th triangle.

    We construct the $t s/3$ trees as follows. See Figure \ref{odd_min_k6k6} for an example.
First, for $\ell=1$ the trees are defined as follows
    $$Z_{i,1} = \begin{cases}
        (u_i,v_i) &\text{ for } i\in \{1,3\} \\
        (u_3,v_1)(u_2,v_2)&\text{ for } i=2\\
        (u_2,v_1)(u_1,v_2)(u_4,v_3)&\text{ for } i=4\\
        (u_{i-1},v_{1})(u_{i-2},v_2)(u_{i},v_3) &\text{ for } 5\le i \le t
    \end{cases}$$
Furthermore, for each $3\le i \le t$  and  $2 \le \ell \le  s/3$, we define

$$Z_{i,\ell}= (u_{i-1},v_1^\ell)(u_{i-2},v_2^\ell)(u_i,v_3^\ell),$$

and, for $2 \le \ell\le \lfloor s/3\rfloor$, we define

$$Z_{1, \ell}=(u_{t-1},v_2^{\ell-1})(u_t,v_1^{\ell-1})(u_1,v_3^\ell)\qquad
\text{and} \qquad Z_{2,\ell}=(u_t,v_2^{\ell-1})(u_1,v_1^\ell)(u_2,v_3^\ell).$$

 The way the trees $Z_{i,\ell}$ are defined ensures that they are pairwise disjoint.
  We define a  2-colouring on each $V(Z_{i,\ell})$ by taking $(u,v)\in V(Z_{i,\ell})$ and assigning the following 

    $$c((u,v)) = \begin{cases}
        1 \text{ if  }(u,v) \text{ has degree at most 1 and }Z_{i,\ell}\ne Z_{2,1}\\
        1 \text{ if  }(u,v) = (u_2,v_2)\\
        2 \text{ if  }(u,v) = (u_3,v_1)\\
        2 \text{ otherwise }
    \end{cases}$$

 Due to the definition, the colouring $c$ is a proper 2-colouring on each $Z_{i,\ell}$.  
 
\begin{figure}[h!]
    \centering
    \includegraphics[scale=1]{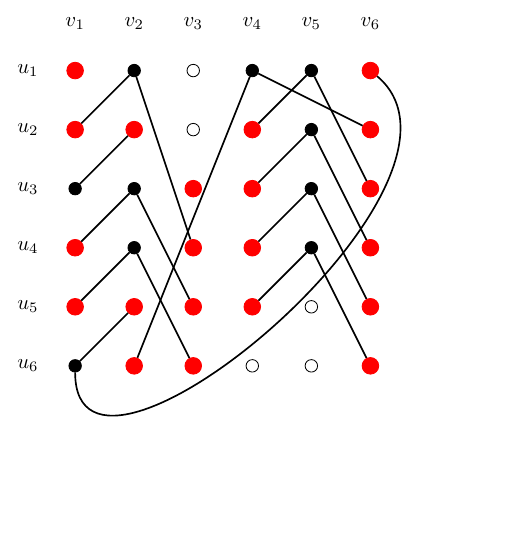} \vspace{-1.5cm}
    \caption{$K_{12}$ as an odd~minor of the graph $K_6\times K_6$. The trees and the 2-colouring are highlighted, where red represents colour 1 and white vertices are not used in the construction of the odd minor. There are two singletons as trees: $(u_1,v_1)$ and $(u_3,v_3)$. }
    \label{odd_min_k6k6}
\end{figure}

It remains to guarantee the existence of a monochromatic edge joining each pair of trees. Let $\mathbf{Z}$ be the union of all the trees $Z_{i,\ell}$. First, note that the two singletons $(u_1,v_1)$, $(u_3,v_3)$  are adjacent, and both are connected to the end vertex $(u_2,v_2)$ of $Z_{2,1}$, which is coloured with $1$. Furthermore, for each $i \in \{1,2,3\}$, the vertex $(u_i,v_i)$ is adjacent to at least one of the end vertices of every $3$-path (coloured with~1 by definition of $c$) in $\mathbf{Z} \setminus \{Z_{1,1}, Z_{2,1}, Z_{3,1}\},$ with the exception of the path $Z_{2,2} = (u_t,v_2)(u_1,v_1^2)(u_2,v_3^2)$. (This is because for every path  $Z\in \mathbf{Z} \setminus \{Z_{1,1}, Z_{2,1}, Z_{2,2}, Z_{3,1}\},$ and every $i\in \{1,2,3\}$,
at least one of the end points of $Z$ does not contain any of  $u_i=u_i^1$  and $v_i=v_i^1$, which is enough to guarantee a monochromatic edge between $(u_i,v_i)$ and $Z$.)
While both $(u_1,v_1)$ and $(u_3,v_3)$ are adjacent to both endpoints of $Z_{2,2}$, but $(u_2,v_2)$ is not adjacent to either endpoint. Yet, we have the monochromatic edge $(u_3,v_1)(u_1,v_1^2)$ joining $Z_{2,1}$ and $Z_{2,2}$.

We now argue that  every pair of 3-paths
$Z, Z'\in \mathbf{Z} \setminus \{Z_{1,1}, Z_{2,1}, Z_{3,1}\}$ is joined by a monochromatic edge.
Note that for every such pair with internal vertices $(u,v), (u',v')$, respectively, satisfying $u\ne u'$ and $v\ne v'$, the edge $(u,v), (u',v')$ exists and it is monochromatic (both vertices coloured 2). Thus, we need to study the situation when 
 $u=u'$ or $v=v'$.

  If we have $v=v'$, then for some $\ell\geq 1$ we either have $v=v_1^{\ell}$ or $v=v_2^{\ell}$. If $v=v_1^{\ell}$, then we have $Z=Z_{1,\ell+1}$ and $Z'=Z_{2,\ell}$, and the edge  $(u_{t-1},v_2^{\ell})(u_2,v_3^\ell)$ is monochromatic and joins $Z$ and $Z'$. Otherwise $Z,Z'$ have internal vertices $(u_i, v_2^{\ell})$,  $(u_j, v_2^{\ell})$, with $i<j$, and are joined by the edge $(u_{i+1}, v_1^{\ell})(u_{j+2}, v_3^{\ell})$  which is monochromatic.

     Otherwise we have  $u=u'$. If $u=u_i \notin \{u_1,u_t\}$, then we have $Z=Z_{i+2,\ell}$ and $Z'=Z_{i+2,\ell'}$ for distinct $\ell, \ell'$. In this case, 
      the edge $(u_{i+1}, v_1^{\ell})(u_{i+2}, v_3^{\ell'})$ is monochromatic and joins $Z$ and $Z'$. If $u=u_t$, then $Z=Z_{1,\ell}$ and $Z'=Z_{1,\ell'}$ for distinct $\ell, \ell'$. Here the edge $(u_{i+1}, v_1^{\ell})(u_{i+2}, v_3^{\ell'})$ is monochromatic and joins $Z$ and $Z'$. 
      
      Finally, if $u=u_1$, we first assume that none of $Z,Z'$ are of the form $Z_{2,\ell}$. Then for distinct $\ell,\ell'$, the path $Z$ has $(u_2,v_1^\ell)$ as an endpoint, while $Z'$ has $(u_2,v_1^{\ell'})$ as an endpoint. The vertex $(u_2,v_1^\ell)$ forms a monochromatic edge with the other endpoint of $Z'$. So we can take $Z=Z_{2,\ell}$ for some $\ell\ge 2$. In this case, $Z'$ has an endpoint of the form $(u_2,v_1^{\ell'})$ or of the form $(u_2,v_3^{\ell'})$. This endpoint forms a monochromatic edge with the endpoint $(u_t,v_2^\ell)$ of $Z.$ The result follows.
\end{proof}

\subsection{Direct product of two general graphs}
\label{sub:2graphs}

In this section we prove Theorem~\ref{thm:directbound}, for which we need the following lemma. We say an expansion is \underline{inflated} if none of its branch trees is a singleton. 


\begin{lemma}\label{arbre_dans_prod}
Let $S$ and $T$ be two trees which are not singletons, then $S \times T$ has at most two components.
\end{lemma}

\begin{proof}

Set $G = S\times T$, and let $c_S$ and $c_T$ be  proper 2-colourings of $S$ and $T$, respectively. Define $A = \{(u,v) \in V(G): (c_S(u), c_T(v)) \in \{(1,1),(2,2)\} \}$ and $B=\{(u,v) \in V(G): (c_S(u), c_T(v)) \in \{(1,2),(2,1)\} \}$. Let us prove the connectivity in $G[A]$; the same proof works for $G[B]$. 
    
    Let us find a path joining an arbitrary pair vertices $(u,v), (v',u')\in A$ . Suppose that $(c(u),c(v)) = (i,i)$ and $(c(u'),c(v')) = (j,j)$ with $i,j\in \{1,2\}$, where $i=j$ is possible.  Let $P$ be the path joining $u$ and $u'$ in $S$ and $P'$ the $(v,v')$-path in $T$. Clearly, the parity of these paths is the same, which is enough for $P\times P'\subseteq S\times T$ to contain a path joining $(u,v)$ and $(u',v')$.
\end{proof}

The next result implies Theorem~\ref{thm:directbound} because every odd expansion of $K_s$ contains an inflated odd expansion of $K_{\lfloor s/2\rfloor}$.

\begin{theorem}
    Let $s,t\geq3$ be integers and $H$ a graph with $\oddhad(H)=3t$. 
If $G$ has an inflated expansion of $K_s$, then we have $\oddhad(G\times H)\geq \oddhad(K_s\times K_t)$.

\end{theorem}
\begin{proof}
Let $\mathbf{S}=\{S_1,S_2,\ldots, S_S\}$ be the branch trees of an inflated odd expansion of $K_{s}$ in $G$, with $c_G$ being its 2-colouring and, for every $ i\neq j$, let $s_is_j$ be a monochromatic edge with endpoints in  $S_i$ and $S_j$, respectively.  Also let $\mathbf{T}=\{T_1,T_2,\ldots,T_{3t}\}$ be the branch trees of an odd expansion of $K_{3t}$ in~$H$, with corresponding 2-colouring $c_H$, and,   for every $i\neq j$,  let $t_it_j$ be a monochromatic edge with endpoints in $T_i$ and $T_j$, respectively. 

For each $p\in \{1,2,\cdots,t\}$, we define $\nabla_p = \{t_{1+3(p-1)}t_{2+3(p-1)}\}\cup \{t_{2+3(p-1)}t_{3p}\}\cup \{t_{3p}t_{1+3(p-1)}\}$, and the graph $C_p=T_{3(p-1)}\cup T_{1+3(p-1)}\cup T_{2+3(p-1)}\cup \nabla_p$. Thus  $C_p$ is the union of three branch trees together with the edges which attest that these trees form a triangle in the odd expansion of~$K_{3t}$ in $H$.
 \begin{claim}
     For $i\in \{1,2,\cdots,s\}$ and $p \in \{1,2,\cdots,t\}$, the graph $S_i\times C_p$ is connected.
 \end{claim}

 \begin{claimproof}
    Consider the tree $C_p'=C_p - \{t_{3p}t_{1+3(p-1)}\}$ and switch the colours given by $c_H$ in $T_{1+3(p-1)}$ and $T_{3p}$, to obtain a proper 2-colouring $c'$ of $C_p'$. Recall that  $c_{G} $ restricted to $S_i$ is a proper 2-colouring. By the proof of Lemma \ref{arbre_dans_prod}, the subgraphs of $ S_i\times C_p'$ induced by $A = \{(u,v) \in V( S_i\times C_p'):(c_G(u),c'(v)) \in \{(1,1),(2,2)\} \}$ and $B = \{(u,v) \in V(S_i\times C_p'):(c_G(u),c'(v)) \in \{(1,2),(2,1)\} \}$, are each connected. Thus it is enough to find an edge between them. But under $c'$ the edge $\{t_{3p}t_{1+3(p-1)}\}$ is still monochromatic. And since the expansion is  inflated, $S_i$ has at least one edge $xy$. Therefore the edge $(x,s_{3p})(y,s_{1+3(p-1)})$ joins $A$ and $B$.
    \end{claimproof}

Consider the  2-colouring that assigns $\phi(u,v) = c_G(u)$ to every $(u,v)\in G\times H$ with $u\in \cup_{1\le i\le s}S_i$. It is not hard to see that we can choose a spanning tree $\mathcal{R}_{i,p}$ of $T_i\times C_p$ such that $\phi$ is a proper 2-colouring of it.

\begin{claim}\label{aretes_entre_RR}
For each pair $\mathcal{R}_{i,p},\mathcal{R}_{j,q}$ with $i\neq j,p\neq q$, there is an edge that joins them and is monochromatic under~$\phi$.
\end{claim}
\begin{claimproof}

There exists $u\in T_i,u'\in T_j$ with $c_G(u)=c_G(u')$. Also there exists $v\in C_p,v'\in C_q$ with $c_H(v)=c_H(v')$.
    Then, $(u,v)(u',v')\in E(G\times H)$ and $\phi(u,v)=c_G(u)=c_G(u')=\phi(u',v')$.
\end{claimproof}

     Set $V(K_s) = \{u_1,u_2\dots,u_s\}$ and $V(K_t) = \{v_1,v_2,\dots,v_t\}$. Let $\mathcal{R}$ be the union of all  $\mathcal{R}_{i,p}$ with  $1 \leq i \leq s$ and $1\leq p \leq t$,
    and $f:   {\mathcal{R}} \rightarrow V(K_s\times K_t)$ be such that $f(\mathcal{R}_{i,p}) =(u_i,v_p)$.
 Note that $f$ is a bijection. 

Denote $m = \oddhad(K_s\times K_t)$. Let $\mathbf{Q} = \{Q_1,Q_2,\dots,Q_m\}$ be trees that form an odd  expansion of $K_m$ in $K_s \times K_t$, let $c_K$ be the 2-colouring witnessing the odd expansion, and let $e_{ij} \in E(K_s \times K_t)$ be the monocromatic edge between $Q_i$ and $Q_j$ for all $i<j\in \{1,\ldots,m\}$.

We now build the odd expansion of $K_m$ in $G\times H$, starting with the branch trees. For each $1 \leq k\leq m$, we define the tree $Z_k$ which is obtained from the forest
$$\bigcup_{x\in Q_k} f^{-1}(x) \subseteq \mathcal{R}$$
by adding the following edges of  $G\times H$: if
$(u_i,v_j)(u_{i'},v_{j'}) \in E(Q_k)$, then pick the monochromatic edge joining $f^{-1}((u_i,v_j))$ and $f^{-1}((u_{i'},v_{j'}))$, which exists by Claim \ref{aretes_entre_RR}.

For the 2-colouring of the odd expansion of $K_m$ in $G\times H$, we define $\gamma$ over $V(Z_1\cup Z_2\cup \dots \cup Z_m)$ according to the following rule: if $(a,b)\in \mathcal{R}_{i,p} $ and we have $\mathcal{R}_{i,p} = f^{-1}((u_i,v_p))$ for some $(u_i,v_p)\in \cup_{1\le k\le m} Q_k$, then we set

$$\gamma((a,b))=\begin{cases}\phi((a,b))&\text{ if }c_K((u_p,v_i))=1,\\
\overline{\phi}((a,b))&\text{ if }c_K((u_p,v_i))=2.
\end{cases}$$
where $\overline{\phi}$ is obtained by switching the colours of $\phi$. 

Let us check that $\gamma$ is a proper colouring on $V(Z_k)$ for each $k\in \{1,\ldots, m\}$. On the one hand, if two adjacent vertices $y$, $y'$ are in the same tree $\mathcal{R}_{i,p}$, then its colours under $\gamma$ are different since both $\phi$ and $\overline{\phi}$ are proper colourings for any tree $\mathcal{R}_{i,p}$. On the other hand, if $y$, $y'$ are in distinct trees 
$\mathcal{R}_{i,p}$, $\mathcal{R}_{i',p'}$, then $c_K(f(\mathcal{R}_{i,p}))\neq c_K(f(\mathcal{R}_{i',p'}))$, and hence, as $yy'$ was monochromatic under $\phi$, vertices $y, y'$ get distinct colours under $\gamma$.  

To finish the proof, we are only left to show that there is a monochromatic edge between every pair of trees $Z_i,Z_j$. Let $Q_i,Q_j\in \mathbf{Q}$ be such that $Z_k$ is obtained from $\cup_{x\in Q_i} f^{-1}(x)$ and $Z_j$ from $\cup_{x\in Q_j} f^{-1}(x)$.  There are vertices $x\in V(Q_i),x'\in V(Q_j)$ such that $xx'\in E(K_s\times K_t)$ and $c_K(x)=c_K(x')$. Hence, if the proper colourings of $f^{-1}(x)$ and $f^{-1}(x')$ under the $\gamma$ are either both $\phi$ or both $\overline{\phi}$. In both cases we apply Claim \ref{aretes_entre_RR} to get a monochromatic edge under $\gamma$.
\end{proof}


\section*{Acknowledgments}

  H. Echeverría is supported by ANID BECAS/DOCTORADO NACIONAL 21231147;  
  A. Jim\'enez is  supported by  ANID/Fondecyt Regular 1220071 and ANID-MILENIO-NCN2024-103;
  D.A. Quiroz is supported by ANID/Fondecyt Regular 1252197 and MATH-AMSUD MATH230035;
  M. Yépez is supported by ANID BECAS/DOCTORADO NACIONAL 21231444.

\bibliographystyle{plain}
\bibliography{ref}

\end{document}